\documentclass[amsthm]{elsart}

\usepackage{graphicx}
\usepackage{amssymb}
\usepackage{amsmath}
\usepackage{mathrsfs}

\begin{document}

\begin{frontmatter}

\title{Quartic monoid surfaces with maximum number of lines}

\author{Mauro Carlo Beltrametti\thanksref{GNS}}
  \address{Dipartimento di Matematica, Universit\`a di Genova, 
Via Dodecaneso, 35, 16146 Genova, Italy}
\ead{beltrame@dima.unige.it}
\author{Alessandro Logar\thanksref{FRA}}
\address{Dipartimento di Matematica e Geoscienze,
  Universit\`a degli Studi di Trieste, Via Valerio 12/1, 34127 Trieste, Italy.}
\ead{logar@units.it}
\author{Maria Laura Torrente}
  \address{Dipartimento di Economia, Universit\`a di Genova, 
Via Vivaldi 5, 16146 Genova, Italy}
\ead{marialaura.torrente@economia.unige.it}

\thanks[GNS]{Partially supported by project G.N.S.A.G.A. 2016, 2017
Geometria algebrica e algebra commutativa}
\thanks[FRA]{Partially supported by the FRA 2018 grant
  ``Aspetti geometrici, topologici e computazionali delle variet\`{a}'',
    Universit\`{a} di Trieste}

\begin{abstract}
  In 1884 the German mathematician Karl Rohn published a substantial
  paper on \cite{ROH} on
  the properties of quartic surfaces with triple points,
  proving (among many other things) that the maximum number of
  lines contained in a quartic monoid surface is $31$.

  In this paper we study in details this class of surfaces. We prove
  that there exists an open subset $A \subseteq \mathbb{P}^1_K$
  ($K$ is a characteristic zero field) that parametrizes (up to a
  projectivity) all the quartic monoid surfaces with $31$ lines;
  then we study the action of $\mathrm{PGL}(4,K)$ on these surfaces,
  we show that the stabiliser of each of them is a group isomorphic
  to $S_3$ except for one surface of the family, whose stabiliser
  is a group isomorphic to $S_3 \times C_3$. Finally we show that 
  the $j$-invariant allows one to decide, also in this situation,
  when two elements of $A$ give the same surface up to a projectivity.
  
  To get our results, several computational tools, available
  in computer algebra systems, are used.
\end{abstract}

\end{frontmatter}

\section{Introduction}
Algebraic quartic surfaces are a classical subject of algebraic geometry
and the study of their rich properties has been developed in research
papers and books since the XIX century. Many
different classifications for several classes of quartic surfaces
where introduced (for instance, see~\cite{ROH}, the books~\cite{es}
and~\cite{cn}; for a more complete discussion,
see~\cite{gas},~\cite{pbpt} and the references given there). 

In more recent years, many classical results have been reconsidered
and presented in a modern language. In particular, a big effort 
has been dedicated to
study the possible singularities on quartic surfaces \cite{Deg},
the number of lines contained in quartic surfaces \cite{gas},
and to study as well the characteristics of monoid surfaces,
as in~\cite{tk1},~\cite{tk2}, where an explicit
description of them in terms of equations can be found
(see also~\cite{jp} and~\cite{bcgm}). 
Unlike the well-known case of smooth cubic surfaces, which contain
$27$ lines, the generic quartic surface does not contain lines.
However there are classes of quartic
surfaces which do contain lines, and it has been
shown that the maximum number is
$64$ (see the paper \cite{gas} and its references).

An interesting class of quartic surfaces is given by the quartic monoid
surfaces, that is, surfaces with a triple point. Their 
classification in terms of other singularities is given in \cite{tk1} and
\cite{tk2}, while in the forthcoming paper \cite{blt}, it is described
a classification according to the possible configuration of lines they
can contain. In particular, it is shown that a quartic monoid surface
can contain at most $31$ lines; indeed such a result was already obtained by
Rohn in a paper published in 1884 (see \cite{ROH}).

In this paper we aim to study into details the class of quartic monoid
surfaces with $31$ lines. In particular, we want to describe the natural
action of the group $\mathrm{PGL}(4, K)$ on these surfaces, which means
that we are interested in quartic monoid surfaces up to projectivity.
We show that quartic monoid surfaces with $31$ lines
can be parametrized by an open subset $A$
of $\mathbb{P}^1_K$ and, moreover, that for any given points
$a, b \in A$
the corresponing surfaces $Q(a)$ and $Q(b)$ are projectively equivalent
if and only if $a$ and $b$ have the same $j$-invariant.
Furthermore, if $a\in A$ is not a primitive root of $-1$, then the stabiliser
of the corresponding surface $Q(a)$ is a group isomorphic to the
group of permutations $S_3$ of three elements
(which does not depend on the parameter $a$),
while if $a$ is a primitive root of $-1$, then the stabiliser of the
corresponding surface is a group with $18$ elements, isomorphic
to $S_3 \times C_3$, where $C_3$ denotes the cyclic group of order $3$.
If the parameter $a$ is rational, then all the
$31$ lines of $Q(a)$ are rational (i.e. obtained by joining
two points with rational coordinates). 
The approach we have followed is quite constructive so that we get (or we
can easily obtain) the explicit equations of all the considered
geometric objects. We have therefore intensively used symbolic
computation tools, and precisely the computer
algebra systems CoCoA~\cite{CoCoA-5} and Sage~\cite{sage}.

\section{Basic properties}
Let $K$ be a characteristic zero field. 
By a \emph{quartic monoid of $\mathbb{P}_K^3$} we mean a quartic surface
of $\mathbb{P}_K^3$ which has a triple point (that, w.l.o.g.\ can be assumed
the origin $O=(0, 0, 0, 1)$). Hence the polynomial defining a quartic
monoid is of  the form:
\begin{equation}
Q(x, y, z, t) = t F_3(x, y, z) + F_4(x, y, z)
\end{equation}
where $F_3, F_4 \in K[x, y, z]$ are two forms of degree
three and four, respectively. From now on, for shortness,
we also denote by $Q$ 
the quartic surface of equation $ Q(x, y, z, t)=0$. 

Let $\pi$ be the plane of $\mathbb{P}^3_K$ of equation
$t=0$. We consider on this plane the two curves of equation $F_3=0$
and $F_4=0$, which intersect in $12$ points $P_0, \dots, P_{11}$ on $\pi$.
The quartic
monoid $Q$ contains the $12$ lines given by $O+P_i$. Moreover, if a line
$r$ is contained in the surface $Q$ and does not pass through the origin,
then its projection from $O$ to the plane $\pi$ gives a line 
joining three of the points $P_i$'s, see \cite{gas} and \cite{ROH}
(indeed, the intersection
$\{r+O\} \cap Q$ is a quartic planar curve which splits into $r$
and three other lines through the origin, which represent
the three possible collinearities
of three points of $\pi$). This remark has
an immediate consequence: the number of lines a quartic monoid can contain
depends on the collinearities of 12 points of the plane. In particular,
Rohn claimed in~\cite{ROH} that the maximum number of collinear triplets of
points among $12$ points in the plane is $19$,  so that the maximum number of
lines contained in a quartic monoid is $31$ ($19$ from the collinearities,
$12$ from the lines through the singular point). He also gave an explicit
(very elegant) equation of a quartic
surface with $31$ lines, which can be expressed (according to
the formulation of~\cite{gas}) in the form:
\begin{equation}
t\left((x+y+z)^3 +xyz\right)+(x+y+z)(x-y)(y-z)(z-x)
\end{equation}

In the present paper, we want to study into details the quartic monoid
surfaces with maximum number of lines, with particular effort to describe
their possible symmetries.

First of all, the combinatorial problem of finding the maximum number
of collinear triplets of $12$ points of the plane can easily be solved with
the help of a computer. The solution we get (unique, up to permutations
of the labels) is the following (see also Figure~\ref{figA}) (where,
from now on, $(i, j, k)$ is a shortcut to denote the triplet $(P_i, P_j, P_k)$):
\begin{equation}
\begin{array}{l}
(0, 1, 2), (0, 3, 4), (0, 5, 6), (0, 7, 8), (0, 9, 10), (1, 3, 7), (1, 4, 5),\\
(1, 6, 8), (1, 10, 11), (2, 3, 5), (2, 6, 7), (2, 9, 11), (3, 6, 10),\\
(3, 8, 11), (4, 6, 9), (4, 7, 11), (4, 8, 10), (5, 7, 10), (5, 8, 9)
\end{array}
\label{19all}
\end{equation}

\begin{figure}
  \begin{center}
  \includegraphics[height=7cm]{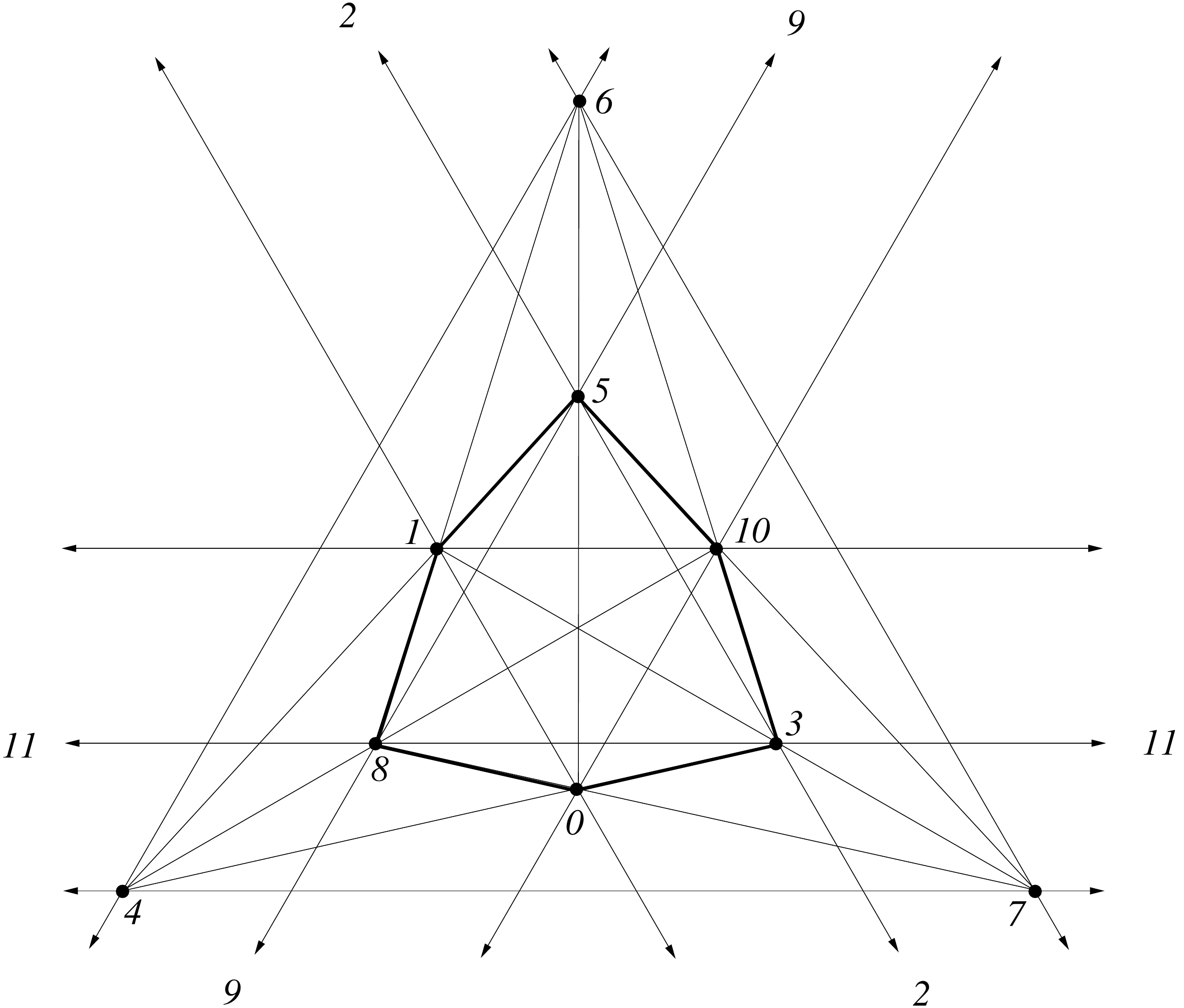}
  \end{center}
  \caption{\small The $19$ collinearities of the
    points described by (\ref{19all}).
    They can be realized starting from 
    the $15$ diagonals of an hexagon (of vertices $0, 3, 10, 5, 1, 8$)
    that meet in the further points $4, 6, 7$ and  $2, 9, 11$
    (these last points lie on the line at infinity).}
\label{figA}
\end{figure}

In order to make computations, it is necessary to assign coordinates to
the points. 
Up to a projectivity of the plane $\pi$, we can assume that 
the coordinates $(x, y, z, t)$ of the points $P_0$, $P_1$, $P_2$, $P_3$, $P_4$
are the following:
\[
P_0,P_1, P_2, P_3, P_4 = (0, 0, 1, 0), (1, 0, 1, 0), (2, 0, 1, 0), 
(0, 1, 1, 0), (0, 2, 1, 0)
\]
The remaining points have to satisfy two constraints: they must satisfy
conditions~(\ref{19all}) and they have to lie on a plane cubic curve. 
Consequently, the only possible coordinates of the remaining points
turn out to be expressed in terms of  a free parameter $a \in K$ as follows:
\[
\begin{array}{l}
P_5, P_6, P_7 = (1, 1, 3/2, 0), (1, 1, 1/2a + 2, 0), (a + 1, 1, a + 2, 0),\\
P_8, P_9 =  (a + 1, 1, 3/2a + 2, 0), (1, -a + 1, 2, 0),\\
P_{10}, P_{11} = (1, -a + 1, -1/2a + 2, 0), (a + 1, -a + 1, 1/2a + 2, 0)
\end{array}
\]

An easy computation shows the following fact.
\begin{lem}
  The points $P_0, \dots, P_{11}$ are all distinct if and only if $a\not= -1$
  or $a \not= 0$. If the $12$ points are distinct, then they satisfy  the
  collinearities conditions  given by~(\ref{19all}) and no more if and only if
   $a\not= 1/2$, $a\not= 1$ or $a\not= 2$.
\end{lem}

We denote by $D$ the set $\{-1, 0, 1/2, 1, 2 \}\subseteq K$ corresponding to
the degenerate cases of the points. 

For a given $a\in K\setminus D$, the cubic curve $C_3(a)$ passing through
the $12$ points $P_0, \dots, P_{11}$ has equation $F_3(a) = 0$, where:
\begin{eqnarray}
  F_3(a) & = & (a - 1)x^3+(2a^2 + 6a - 1)x^2y-3(a-1)x^2z+(a^2 + 4a + 1)xy^2\\
  \nonumber
  & & +  2(a - 1)xz^2 -2a(a+4)xyz+(a + 1)y^3-3(a+1)y^2z+2(a+1)yz^2\nonumber
\end{eqnarray}

In order to construct the quartic monoids, we need quartic curves
of the plane $\pi$  passing
through the $12$ points. For a fixed $a\in K \setminus D$,
let $\mathcal{L}_a$ be the linear system of quartic curves through
the $12$ points.

\begin{lem}
  The linear system $\mathcal{L}_a$ of quartic curves through the $12$
  points $P_0, \dots, P_{11}$ has dimension $4$. If $\ell_1$, $\ell_2$, $\ell_3$
  are three fixed, generic lines in the plane and if $r_1$, $r_2$, $r_3$, $r_4$
  are the four lines passing through the triplets of collinear points
  $(P_0, P_1, P_2), $ $(P_3, P_6, P_{10}), $ $(P_4, P_7, P_{11}), $
  $(P_5, P_8, P_9)$, then a basis for the
  linear system $\mathcal{L}_a$ is given by the four quartics:
  \begin{equation}
  C_3(a)+\ell_1, \quad C_3(a)+\ell_2, \quad C_3(a)+\ell_3,
  \quad r_1+r_2+r_3+r_4
  \label{4quart}
  \end{equation}
\end{lem}

\begin{proof}
  To compute the dimension of the linear system, it is enough to take a generic
  quartic curve of $\mathbb{P}^3$ and to impose that it passes through the
  $12$ points. The computation of the rank of the $12\times 15$
  matrix associated to the
  system, gives the dimension of $\mathcal{L}_a$. In order to speed up the
  computation, it is convenient to observe that the points $P_0, \dots, P_5$
  have fixed coordinates. Therefore, to find the dimension of $\mathcal{L}_a$
  we can solve six of the equations in terms of the others, hence it
  suffices to find the rank of a $6\times 9$ matrix (whose entries are
  polynomials in $a$). It turns out that all the maximal minors are zero,
  while the $5 \times 5$ minors are zero if and only if $a=0$, which is
  excluded by our hypothesis. Finally, all the four quartic curves
  of~(\ref{4quart}), contain the $12$ points and they 
  are clearly linearly independent. 
\end{proof}
The equation of the quartic curve which splits into the lines
$r_1, \dots, r_4$ is:
\begin{eqnarray*} 
F_4(a) &=& y(3ax - 2az + x - y)(2ax + ay - 2az + 3x + y - 2z)\\ 
& & (ax + 2x + 2y - 2z)
\end{eqnarray*}
Thus, as a consequence of the previous lemma one sees  that all
the quartic monoids with $31$ lines (up to a change of coordinates) have
equation:
\[
 tF_3(a) + (\alpha_0x+\alpha_1y+\alpha_2z)F_3(a) + bF_4(a) = 0
\]
where $a$, $\alpha_0$, $\alpha_1$, $\alpha_2, b$ are parameters in $K$. If
we change the coordinates writing $t-\alpha_0x-\alpha_1y-\alpha_2z$
in place of $t$, we finally get that, up to a change of coordinates,
all the quartic monoids with $31$ lines have equation $Q(a, b) = 0$,
where:
\begin{equation}
  Q(a, b) = tF_3(a) + b F_4(a)
  \label{Qab}
\end{equation}
with $a \in  K \setminus D$, $b \not = 0$. Moreover,  starting from the
same given points $P_0, \dots, P_{11}$, different values of
$b$ give different quartic monoids. 

We conclude this section with the following proposition, which shows
that the surface~(\ref{Qab}) does not have singular
points (apart from the origin).

\begin{prop}
  If $a \not\in D$ and $b \not= 0$, then the only singular point
  of the quartic surface~(\ref{Qab}) is the origin. 
\end{prop}
\begin{proof}
  If $T$ is a singular point of $Q = Q(a, b)$ different from the origin $O$,
  then the line
  joining $O$ and $T$ is contained in $Q$, therefore every singular
  point of $Q$ lies on one of the $12$ lines of $Q$ passing through
  the origin.  Hence we have to look for singular points among the points
  $T_\lambda^{(i)} = O + \lambda P_i$ (where $\lambda$ is a parameter).
  If we substitute $T_\lambda^{(i)}$ into the gradient of $Q$ and we saturate
  the ideal w.r.t.\ the polynomial $a(a-1)(a+1)(a+2)(a+1/2)b$ we
  always get the ideal $(\lambda^2)$ which means that $T_\lambda^{(i)}$
  cannot be singular. 
\end{proof}

\section{Convergent sextuples}

In this section we consider a special configuration of points.
\begin{defn}\label{CS}
Let $\mathcal{E} = (E_0, \dots, E_5)$ be six points in 
$\mathbb{P}^3$ satisfying the following properties.
\begin{itemize}
\item All the points are distinct and all are different from the origin $O$;
\item The three lines $E_0+E_1$, $E_2+E_3$, $E_4+E_5$ meet in a common,
  new point $A$;
\item There are no other collinearities among the points;
\item The six points are not contained in a plane.
\end{itemize}
The sextuple $\mathcal{E}$ will be called a \emph{convergent sextuple}.
\end{defn}
A simple example of a convergent sextuple is the following:
\begin{eqnarray}
\left((1, 0, 0, 1),\;
 (2, 0, 0, 1),\;
 (0, 1, 0, 1),\;
 (0, 2, 0, 1),\;
 (0, 0, 1, 1),\;
(0, 0, 2, 1) \right)
\label{csBase}
\end{eqnarray}

Convergent sextuples fulfil the following property.
\begin{prop}
If $\mathcal{E}$ and $\mathcal{F} = (F_0, \dots, F_5)$ are 
two convergent sextuples, then there exists precisely
one projectivity which sends $E_i$ to $F_i$, $i = 0, \dots, 5$.
\label{unicaProj}
\end{prop}
\begin{proof}
  We call $R_1$ the intersection point of the lines $E_0 + E_3$
  with $E_1+E_2$,  $R_2$ the intersection point of the lines
  $E_0 + E_5$ with $E_1+E_4$, and $R_3$ the intersection point of
  $E_2+E_5$ with $E_3+E_4$. Then it is possible to verify that
  the lines $R_1+E_4$, $R_2+E_2$ and
  $R_3+E_0$ meet in a common point $A_1$. Let $A$ be the intersection
  point of the lines $E_0+E_1$, $E_2+E_3$, $E_4+E_5$
  as in Definition \ref{CS}. Then the five points
  $A$, $E_0$, $E_2$, $E_4$, $A_1$
  are in generic position. Let now $S_1 = (F_0+F_3) \cap (F_1+F_2)$,
  $S_2 = (F_0+F_5) \cap (F_1 + F_4)$ and $S_3 = (F_2+F_5)\cap (F_3+F_4)$.
  Then, as above, the lines  $S_1+F_4$, $S_2+F_2$ and
  $S_3+F_0$ meet in a common point $B_1$. Let $B$ the intersection point
  of $F_0+F_1$, $F_2+F_3$, $F_4+F_5$. Also now the points $B$, $F_0$,
  $F_2$, $F_4$, $B_1$ are in generic position, then, from the
  fundamental theorem of projectivities, there exists exactly one
  proiectively which sends $A$ to $B$, $E_0$ to $F_0$, $E_2$ to $F_2$,
  $E_3$ to $F_3$ and $A_1$ to $B_1$. It is easy to see that, consequently,
  $E_1, E_3, E_5$ are also sent to, respectively, $F_1, F_3, F_5$.
\end{proof}

\begin{rem}
  It is worthy of note a nice geometric property of a convergent sextuple:
  if we compute in the same way the point
  $A_2$ as the intersection of the three lines $R_1+E_5$, $R_2+E_3$ and
  $R_3+E_1$, then the points $A$, $A_1$, $A_2$ are collinear. 
\end{rem}

In every quartic surface $Q$ defined by~(\ref{Qab})
we can find convergent sextuples $\mathcal{E}$ with the further
condition that 
the six lines $(E_0+E_1, E_2+E_3, E_4+E_5, E_0+E_2, 
E_3+E_5, E_1+E_4)$ are contained in the surface. To see this, we can
proceed as follows. 
Take three points $P_i$, $P_j$, $P_k$ among the $12$ points of intersection
of $F_3=0$ and $F_4=0$, with the 
following properties:
\begin{itemize}
\item $P_i$, $P_j$, $P_k$ are not collinear (i.e. $(i, j, k)$ is not in 
the list~(\ref{19all}));
\item for each of the three couples $(P_i, P_j)$, $(P_i, P_k)$, $(P_j, P_k)$
we can find a point $P_{ij}$, $P_{ik}$, $P_{jk}$ such that the 
triplets $(P_i, P_j, P_{ij})$, $(P_i, P_k, P_{ik})$, $(P_j, P_k, P_{jk})$
are in~(\ref{19all}).
\end{itemize}
The plane $O+P_i+P_j$ intersects $Q$ along the three lines $O+P_i$, 
$O+P_j$, $O+P_{ij}$ and a further line $\ell_{ij}$ which does not pass through
$O$. Similarly, we find a line $\ell_{ik}$ and a line $\ell_{jk}$ of $Q$. 
Then a convergent sextuple is given by the points: 
\begin{equation}
\begin{array}{ll}
E_0 = (O+P_i) \cap  \ell_{ij}, & E_1 = (O+P_i) \cap \ell_{ik},
\\ 
E_2 = (O+P_j) \cap \ell_{ij}, & E_3 = (O+P_j) \cap \ell_{jk}, \\
E_4 = (O+P_k) \cap \ell_{ik}, & E_5 = (O+P_k) \cap \ell_{jk}.
\end{array}
\label{sexExc}
\end{equation}
A convergent sextuple of this form is called a \emph{standard
  convergent sextuple}. 
For example, we can choose for $i$, $j$, $k$,  the indices $0$, $1$, $3$
(in this order). In this 
case the points $E_0, \dots, E_5$ are:
\begin{equation}
\begin{array}{l}
  (0, 0, 1, 0), \; (0, 0, 1, 4ab), \; (1, 0, 1, 0),\;
  (1, 0, 1, -(a + 1)b),\\
  (0, 1, 1, 2(2a + 1)b), \;(0, 1, 1, (2a + 1)b)
\end{array}
\label{bcs013}
\end{equation}

The relevance of the standard convergent sextuples is given by the following
proposition, whose proof is an immediate consequence of
Proposition~\ref{unicaProj}.

\begin{prop}
  Suppose we have two quartic surfaces $Q = Q(a, b)$ and $Q' = Q(a', b')$
  of the family~(\ref{Qab}), for
  $a, a', b, b' \in K$. Then the two quartic surfaces are
  projectively equivalent if and only if we can find a standard convergent
  sextuple $\mathcal{E}$ on $Q$ and a standard convergent sextuple $\mathcal{E}'$
  on $Q'$ such that the projectivity which sends $\mathcal{E}$ into
  $\mathcal{E}'$ (according to Proposition~\ref{unicaProj}) sends $Q$ to
  $Q'$.
\end{prop}

Let us point out another consequence of
  Proposition~\ref{unicaProj}. Fix
a convergent sextuple $\mathcal{E}_0$, as in~(\ref{csBase}), and take a
basic convergent sextuple $\mathcal{E}$ (as in~(\ref{bcs013}))
on the quartic surface $Q(a, b)$. Then the projectivity which sends
$\mathcal{E}$ to $\mathcal{E}_0$ transforms the quartic monoid
$Q$ into a quartic monoid
passing through the convergent sextuple~(\ref{csBase}). In other
words, we can assume that all the quartic monoids pass
through the sextuple~(\ref{csBase}),
and, moreover, that such a convergent sextuple  is a standard one.

The relevant fact is that in this way the equation of
the transformed quartic monoid has a new equation which does not contain
the parameter $b$ (as soon as we assume it is not zero). Therefore, all
quartic surfaces of~(\ref{Qab}) obtained by different
values of $b$ are projectively equivalent. 

Here is the family of polynomials giving the quartic
monoids  for which~(\ref{csBase}) is a basic convergent sextuple:
\begin{eqnarray}
  \label{Qa}
  Q(a) & = & 2ax^3y+ 4ax^2y^2+ 2axy^3 -(2a - 1)(a - 2)x^3z+\\ \nonumber
  & & -(5a^2 - 17a + 5)x^2yz -3(a^2 - 4a + 1)xy^2z + ay^3z+\\ \nonumber
  & & -3(2a - 1)(a - 2)x^2z^2 -(7a^2 - 19a + 7)xyz^2+ 2ay^2z^2+\\ \nonumber
  & & -2(2a - 1)(a - 2)xz^3 +ayz^3 -2atx^2y-2atxy^2+ 2(2a - 1)(a - 2)tx^2z+ \\ \nonumber
  & & 4(a^2 - 3a + 1)txyz -2aty^2z+ 2(2a - 1)(a - 2)txz^2 -2atyz^2 \nonumber
\end{eqnarray}
The polynomial  $Q(a)$  is obtained from the basic convergent
sextuple of~(\ref{Qab})
constructed from the triplet $(11, 10, 9)$, since this choice allows us to
obtain a simpler equation for~$Q(a)$. Clearly, we assume
$a \not \in D= \{-1, 0, 1/2, 1, 2 \}$
so that the surface $Q(a)$ is smooth outside the origin and its $31$ lines
do not degenerate.
We have therefore the following:

\begin{thm}
  Let $K$ be a field of characteristic zero and let
  $a\not \in \{-1, 0, 1/2, 1, 2 \}$. Then
  all quartic monoid surfaces of $\mathbb{P}^3_K$  with no other
  singularities outside of the the triple point
  and with the maximum number of lines
  are projectively equivalent to a surface of equation
   (\ref{Qa}).
  \label{teor1}
\end{thm}
Finally, the points $\bar{P}_0, \dots, \bar{P}_{11}$ of
intersection of the $12$ lines of $Q(a)$ passing through
the origin with the plane $t=0$ have now the following coordinates:
\begin{equation}
  \begin{array}{l}
    (0, 1, -1, 0),\ (1, -1, 0, 0),\ (1, 0, -1, 0),\ (1, 2a - 1, -a, 0),
    \ (1, a - 2, 1, 0),\\
    (a, -2a + 1, -1, 0),\ (1, a - 2, -a, 0), \ (a, -a + 2, -1, 0), \ 
    (a, -2a + 1, a, 0),\\ \ (0, 0, 1, 0), \ (0, 1, 0, 0),\ (1, 0, 0, 0)
  \end{array}
  \label{Pbar}
\end{equation}
They still respect the collinearities given by (\ref{19all}).

\section{Stabilisers}

Starting from relations~(\ref{19all}), it is possible to construct
all the triplets of points $(i, j, k)$
with the property introduced in the previous section,
i.e.\ such that the sextuple $\mathcal{E}$ given by~(\ref{sexExc})
is standard and convergent in the sense of Definition \ref{CS}.
Note that the number of the possible triplets
is $720$ (indeed, any set of points $\{i, j, k\}$ gives six standard
convergent sextuples,
one for each permutation of the indices $i, j, k$).

Let $\mathcal{M} \in \mathrm{PGL}\,(4, K)$.
We consider the natural action of the matrix $\mathcal{M}$
on $\mathbb{P}^3_K$:
given a point $R$ in $\mathbb{P}^3_K$ of
coordinates $(u_0, u_1, u_2, u_3)$, we
send it to the point $\mathcal{M}\cdot R$, whose coordinates are
(the transpose
of) $\mathcal{M} \cdot {}^t(u_0, u_1, u_2, u_3)$.
Hence $\mathrm{PGL}\,(4, K)$ also acts on
homogeneous polynomials in a natural way: if $F(x, y, z, t)$ is a
homogeneous polynomial, then $\mathcal{M}\cdot F$ is the polynomial
$F\left(\mathcal{M}^{-1} \cdot {}^t(x, y, z, t)\right)$.
According to this definition, if a
point $R\in \mathbb{P}^3$
is a zero of $F$, then the point $\mathcal{M}\cdot R$ is a zero of
$\mathcal{M} \cdot F$. 

Let us consider the following problem: 
given a quartic monoid $Q(a)$ ($a \in K \setminus D$),
compute its stabiliser w.r.t.\ the action of
$\mathrm{PGL}\,(4, K)$, i.e.\ the group
(in the following, by $\mathcal{M} \cdot Q(a) = Q(a)$
we mean that the 
surfaces defined  by the equations $Q(a)=0$ and 
  $\mathcal{M}\cdot Q(a) = 0$, respectively,  coincide)
\[
  G_a = \mathrm{Stab}_{\mathrm{PGL}(4, K)}Q(a) = 
\{\mathcal{M} \in \mathrm{PGL}\, (4, K) \ | \
    \mathcal{M} \cdot Q(a) = Q(a)\}
\]
The quartic monoid surface $Q(a)$ has the convergent sextuple
$\mathcal{E}_0$ as in~(\ref{csBase}) among
its standard convergent sextuples. If $\mathcal{E}$ is
another standard convergent
sextuple of $Q(a)$, let $\mathcal{M}$ be the unique (up to a
multiplicative constant) matrix such that
$\mathcal{M} \cdot \mathcal{E}_0 = \mathcal{E}$ (see
Proposition~\ref{unicaProj}).
If we compute $Q_\mathcal{M} = \mathcal{M}\cdot Q(a)$, we
have to select those matrices $\mathcal{M}$ such that
$Q_\mathcal{M}$ and $Q(a)$ define the same surface
(moreover, we can also check if there are some
specific values of $a$ such that for that specific value
$Q_\mathcal{M}$ and $Q(a)$ coincide).
Of course, this construction has to
be repeated for the $720$ standard convergent sextuples $\mathcal{E}$
of $Q(a)$.
The computations are not too hard; the obvious way to see when
$Q_\mathcal{M}$ and $Q(a)$ define the same surface, is to construct the
$2\times 35$ matrix $W = (w_{rs})$ whose rows are, respectively,
the coefficients of $Q_\mathcal{M}$
and of $Q(a)$ w.r.t.\ the $35$ monomials of degree $4$ in the variables
$x, y, z, t$. The two surfaces coincide if and only if the rank of
this matrix is one. In principle, the number of minors ($595$) might
seem problematic, but the computation can
greatly be reduced by the following observations: (i) in the matrix $W$
we can erase all the columns in which both entries are 0;
(ii) if in the matrix $W$
the element $w_{1s}$ is not zero but the element $W_{2s}$ is zero,
then the two surfaces cannot coincide and no other computations
are needed; (iii) if $w_{1s} = 0$,
then among the equations which test the coincidence of the two surfaces,
we have to add $w_{2s}=0$.

The following result holds true.

\begin{thm} If $a\in K \setminus \{-1, 0, 1/2, 1, 2\}$ is not a root
  of $x^2-x+1=0$,
  then the stabiliser $G_a$ of $Q(a)$ is isomorphic to the group $S_3$,
  and is generated by the following two matrices:
\[
\left(\begin{array}{rrrr}
2 & 2 & \phantom{-}2 & 0 \\
0 & -2 & 0 & 0 \\
-2 & 0 & 0 & 0 \\
0 & 1 & 3 & -2
\end{array}\right), \qquad
\left(\begin{array}{rrrr}
2 & 2 & 2 & \phantom{-}0 \\
0 & -2 & 0 & 0 \\
0 & 0 & -2 & 0 \\
0 & -2 & -3 & 2
\end{array}\right)
\]
If $a=\varepsilon$ is a solution of the equation $x^2-x+1=0$, then
the equation of the quartic monoid surface $Q(\varepsilon)$
becomes $Q'=0$ where:
\begin{eqnarray*}
  Q' & = & x^3y + 2x^2y^2 + xy^3 + 3/2x^3z + 6x^2yz + 9/2xy^2z+ 1/2y^3z
  + 9/2x^2z^2\\
  & & + 6xyz^2 + y^2z^2 + 3xz^3 + 1/2yz^3
  - x^2yt - xy^2t - 3x^2zt - 4xyzt\\
  & & - y^2zt - 3xz^2t - yz^2t
\end{eqnarray*}
In this case, its stabiliser $G_{\epsilon}$
has order $18$, is isomorphic to
$S_3 \times C_3$
and is generated by the following two matrices:
\[
\left(\begin{array}{rrrr}
2 & 2 & 2 & \phantom{-}0 \\
0 & -2 & 0 & 0 \\
0 & 0 & -2 & 0 \\
0 & -2 & -3 & 2
\end{array}\right), \qquad
\left(\begin{array}{rrrr}
0 & 2 \epsilon - 2 & 2 \epsilon - 4 & 0 \\
0 & -4 \epsilon + 2 & 0 & 0 \\
2 \epsilon - 4 & 2 \epsilon - 2 & 0 & 0 \\
-3 & -3 \epsilon & -3 \epsilon & 4 \epsilon - 2
\end{array}\right)
\]
\label{th4.1}
\end{thm}

\medskip
We finally complete Theorem \ref{teor1} answering the
following question: how many vales of $a$ in (\ref{Qa}) give the same
quartic monoid up to a projectivity?

We can answer the question by mean of computations quite similar to
the previous ones. Let $\mathcal{E}_0$ be the convergent sextuples
given by
(\ref{csBase}) (remember that $\mathcal{E}_0$
is a standard convergent sextuple
for all the surfaces of the family (\ref{Qa})). If $Q(a)$ and $Q(b)$
are two surfaces of the family (\ref{Qa}) which are projectively equivalent,
then there exists a standard convergent sextuple $\mathcal{E}$ of $Q(a)$
such that $\mathcal{M} \cdot Q(a)$ and $Q(b)$ coincide, where $\mathcal{M}$
is the $4\times 4$ matrix which sends $\mathcal{E}$ into $\mathcal{E}_0$.

By repeating the above computation for all the $720$ standard convergent
sextuples and by selecting the cases leading to a positive answer
to the question, we get the following result.

\begin{thm} Let $a, b \in K \setminus \{-1, 0, 1/2, 1, 2\}$
  be two elements and consider
  the two quartic surfaces $Q(a)$ and $Q(b)$ from the family (\ref{Qa}). 
  Then $Q(a)$ and $Q(b)$ are projectively equivalent if and only if $a$
  has one of the following values:
  \begin{equation}
    b, \quad \frac{1}{b}, \quad \frac{1}{1-b}, \quad  \frac{b}{b-1},
    \quad 1-b, \quad \frac{b-1}{b}
  \label{jinv}
  \end{equation}
\end{thm}

\medskip
We conclude by noticing how conditions (\ref{jinv}) show that
the  action of
$\mathrm{PGL}\, (4, K)$ on quartic monoid surfaces
of $\mathbb{P}^3_K$ with maximum number of lines is quite similar
to the action of $\mathrm{PGL}\, (3, K)$ on cubic curves of
$\mathbb{P}^3_K$ (see \cite{JH}, Lecture 10).
In particular, two quartic monoid surfaces $Q(a)$ and $Q(b)$
are projectively equivalent if and only if $a$ and $b$ have the
same $j$-invariant.

\begin{rem}
  Among the lines of $Q(a)$, the lines $O + \bar{P}_2$,
  $O+\bar{P}_9$ and $O+\bar{P}_{11}$ can be
  distinguished from the others: they are the only three lines
  through the singular point of $Q(a)$ which intersect only $4$ other
  lines of $Q(a)$ (see (\ref{19all}) or figure \ref{figA}).
  The group $G_a$ corresponds to the permutations of
  these three lines.
  Analogously, the group $G_{\epsilon}$ corresponds to the permutations
  of the three lines above and the cyclic permutation of the three lines
  $\bar{P}_1 + \bar{P}_3 + \bar{P}_7$,
  $\bar{P}_0 + \bar{P}_5 + \bar{P}_6$,
  $\bar{P}_4 + \bar{P}_8 + \bar{P}_{10}$.  
\end{rem}

\bibliographystyle{plain}
\bibliography{paperBLT}

\end{document}